\documentclass[sn-mathphys]{sn-jnl}

\usepackage{amsmath,amssymb,amsfonts,amsthm}
\usepackage{graphicx}
\usepackage{hyperref}
\usepackage{mathtools}
\usepackage[numbers,square]{natbib}
\theoremstyle{thmstyleone}%
\newtheorem{theorem}{Theorem}
\newtheorem{proposition}[theorem]{Proposition}%
\newtheorem{lemma}[theorem]{Lemma}%
\newtheorem{observation}[theorem]{Observation}%
\newtheorem{corollary}[theorem]{Corollary}%
\theoremstyle{thmstyletwo}%
\newtheorem{problem}{Problem}
\theoremstyle{thmstylethree}%
\newtheorem{definition}{Definition}%

\usepackage{tikz}
\usetikzlibrary{arrows}
\usetikzlibrary{positioning,arrows.meta}

\newcommand{\AX}[1]{{\upshape{(#1)}}}
\newcommand{\smallqed}{{\tiny ($\Box$)}}

\begin{document}
\author[1,2]{\fnm{Bo{\v{s}}tjan} \sur{Bre{\v{s}}ar}}
\email{bostjan.bresar@um.si}
\author[3,4]{\fnm{Manoj} \sur{Changat}}
\email{mchangat@keralauniversity.ac.in}
\author[5]{\fnm{Prasanth G.} \sur{Narasimha-Shenoi}}
\email{prasanthgns@gmail.com}
\author[6,7]{\fnm{Bruno J.} \sur{Schmidt}}
\email{bruno@bioinf.uni-leipzig.de}
\author[6,7,8,9,10]{\fnm{Peter F.} \sur{Stadler}}
\email{studla@bioinf.uni-leipzig.de}

\affil[1]{\orgdiv{Department of Mathematics and Computer Science,
    Faculty of Natural Sciences and Mathematics},
  \orgname{University of Maribor}, 
  \orgaddress{\street{Koro{\v{s}}ka c.\ 160},
  \postcode{2000} \city{Maribor}, \country{Slovenia}}}

\affil[2]{\orgname{Institute of Mathematics, Physics and Mechanics}, 
    \orgaddress{\city{Ljubljana}, \country{Slovenia}}} 

\affil[3]{\orgdiv{Department of Futures Studies}, \orgname{University of
    Kerala}, \orgaddress{\street{Karyavattom}\,
    \postcode{695 581} \city{Thiruvananthapuram}, 
    \country{India}}}
\affil[4]{\orgdiv{Centre for Socio-economic and Environmental Studies (CSES)}, \orgname{Khadi Federation Building}, \orgaddress{\street{NH By-pass}\,
    \postcode{682 024} \city{Ernakulam}, 
    \country{India}}}
\affil[5]{\orgdiv{Department of Matheamtics},
  \orgname{Goverment Collge, Chittur}, \orgaddress{\street{}
    \postcode{678 104} \city{Palakkad}, 
    \country{India}}}
\affil[6]{\orgname{Max Planck Institute for Mathematics in the Sciences},
  \orgaddress{\street{Inselstra{\ss}e 22},
    \postcode{D-04103} \city{Leipzig}, \country{Germany}}}

\affil[7]{\orgdiv{Bioinformatics Group, Department of Computer Science \&
    Interdisciplinary Center for Bioinformatics}, \orgname{Leipzig
    University}, \orgaddress{\street{H{\"a}rtelstra{\ss}e 16–18},
    \postcode{D-04107} \city{Leipzig}, \country{Germany}}}

\affil[8]{\orgdiv{Department of Theoretical Chemistry}, \orgname{University
    of Vienna}, \orgaddress{\street{W{\"a}hringerstra{\ss}e 17},
    \postcode{A-1090} \city{Wien}, \country{Austria}}}

\affil[9]{\orgdiv{Facultad de Ciencias}, \orgname{Universidad National de
    Colombia}, \orgaddress{\city{Bogot{\'a}}, \country{Colombia}}}

\affil[10]{\orgname{Santa Fe Institute}, \orgaddress{\street{1399 Hyde Park
      Rd.}, \city{Santa Fe}, \state{NM} \postcode{87501}, \country{USA}}}

\title{Smooth Graphs}


\abstract{The notion of smoothness was introduced originally in the context
  of step systems on connected graphs. Smoothness turns out to be a very
  general property of metrics defined by a five-point condition. Restricted
  to graphs, it is closely related to the convexity of point-shadows. We show
  that smoothness is preserved by isometric subgraphs, both Cartesian and
  strong graph products, and gated amalgams. As a consequence, median
  graphs and many of their generalizations are smooth. We also show that
  $\ell_1$-graphs are smooth. On the other hand, an induced $K_{2,3}$ or
  $K_{1,1,3}$ is incompatible with smoothness. Finally, we characterize
  smooth graphs among the Ptolemaic graphs as precisely the
  $K_{1,1,3}$-free Ptolemaic graphs.}

\maketitle

\section{Introduction}

Many interesting properties of a connected graph $G=(V,E)$ are closely related to
the metric structure on its vertex set $V$ that is determined by the
shortest path distance $d:V^2\to\mathbb{N}_0$, where $d(u,v)$ is the
minimum number of edges along any $u,v$-path. Several well-studied
relations and set functions derive from $d$. The \emph{neighborhood} of a
vertex can be specified as $N(u)=\{w\in V:\, d(u,w)=1\}$.  The
intervals in $G$ are defined as sets.
\begin{equation}
  I[u,v]\coloneqq \{w\in V:\, d(u,w)+d(w,v)=d(u,v)\}.
\end{equation}
A subset $S$ of $V$ is \emph{geodesically convex} or
$g$ \emph{-convex} if $I[u,v]\subseteq S$ for all $u,v\in S$.  For a subset
$W$ of $V$, the smallest g-convex set containing $W$ is called
the \emph{geodesically convex hull} of $W$, denoted as $\langle W \rangle
$.

The family of functions $I:V^2\to 2^V$ that are the interval functions of
undirected connected graphs can be characterized by a fairly simple set of
first-order axioms \cite{nebe-94,Mulder:09}. A closely related set function
is the \emph{step function} $S:V^2\to 2^V$ of a graph \cite{Ne2} defined by
\begin{equation}
  S(u,v) = I[u,v]\cap N(u). 
\end{equation}
Similar to the interval functions of graphs, their step functions also have
a concise axiomatic characterization \cite{Ne2}. It builds upon the more
general notion of \emph{signpost systems}, which were introduced as ternary
relations that encode path information in a strictly local manner, without
presupposing an underlying graph. Throughout, we write $x\in S(u,v)$
instead of $(u,x,v)$ for a ``signpost'' at vertex $u$ pointing towards $x$
to the first step on a path towards $v$ \cite{Muld1}.  In their most
general form, the signpost system satisfies three natural axioms for all
$u,v,x\in V$:
\begin{itemize}
\item[\AX{A}] If $v\in S(u,x)$ then $u\in S(v,u)$. 
\item[\AX{B}] If $v\in S(u,x)$ then $u\notin S(v,x)$.
\item[\AX{H}] If $u \neq v$ then there exists a $z\in V$ such that
              $z\in S(u,v)$.
\end{itemize}
Each signpost system $S:V^2\to 2^V$ is naturally associated with the
\emph{underlying graph} $G$ with $uv$ being an edge if and only if $v\in S(u,v)$.
The step systems deriving from connected undirected graphs have
characterizations as signpost systems that satisfy several additional first
order axioms, see \cite{Ne2,Nebesky:00,Muld1,Changat:26a} for slightly
different characterizations.  Signpost systems $S$ that satisfy these
properties are called \emph{step systems} and coincide with the step system
of the underlying graph. Thus, step systems are an equivalent description of
connected undirected graphs.
 
In \cite{nebesky2005signpost}, Nebesk{\'y} introduced a ``smoothness
axiom'' for signpost systems as follows: 
\begin{enumerate}
\item[\AX{Sm}] If $v\in S(u,w)$, $v\in S(u,y)$, and $x\in S(w,y)$ then
  $v\in S(u,x)$. 
\end{enumerate}
We remark that the smoothness axiom is stated in a different but equivalent
form in \cite{nebesky2004properties}. Due to the 1-1 correspondence
between step systems and connected undirected graphs, graph classes are
defined by additional properties of signpost systems. Here, we focus on
\AX{Sm}:
\begin{definition}
  A connected graph $G$ is \textit{smooth} if its step system satisfies
  \AX{Sm}.
\end{definition}
Using the connection between signpost triples and the shortest paths in
the underlying graphs, it is straightforward to translate \AX{Sm} into
graph-theoretical terms:
\begin{observation}
  \label{obs:interval}
  A connected graph $G$ is smooth if and only if any five
  vertices $u,v,w,x,y$ satisfy the following condition:
\begin{itemize}
\item[] If $uv,wx\in E(G)$, $v\in I[u,w]\cap I[u,y]$, and $x\in I[w,y]$, then
  $v\in I[u,x]$.
\end{itemize}
\end{observation}	

Recalling that the interval function of a graph satisfies in particular the
betweenness axioms~\cite{mulder1980interval},
\begin{itemize}
\item[\AX{b2}] $x\in I[u,v]$ implies $I[u,x]\subseteq I[u,v]$, and 
\item[\AX{b3}] $x\in I[u,v]$ and $y\in I[u,x]$ implies $x\in I[y,v]$,
\end{itemize}
we observe that the smoothness condition \AX{Sm} is trivially satisfied
whenever any of the five vertices in its precondition coincide.
\begin{lemma}
  \label{lem:5only}
  Let $G$ be a graph and suppose $X=\{u,v,w,x,y\}$ has cardinality $|X|\le
  4$ with $uv\in E(G)$, $wx\in E(G)$, $v\in I[u,w]\cap I[u,y]$, and $x\in
  I[w,y]$. Then $v\in I[u,x]$.
\end{lemma}
\begin{proof}
  We proceed case by case: (i) $u\in\{v,w,x,y\}$: By assumption $u\ne v$.
  If $u=w$ then $v\in I[u,u]=\{u\}$ and hence $v=u$, a contradiction.  If
  $u=x$ the implication is trivial.  If $u=y$ the $v\in I[u,u]$, again a
  contradiction.
  (ii) $v\in\{w,x,y\}$: If $v=x$ the conclusion is trivially true.  If
  $v=w$ we have $v\in I[u,v]\cap I[u,y]$ iff $v\in I[u,y]$, $x\in I[v,y]$
  and by \AX{b2} $I[v,y]\subseteq I[u,y]$ and thus $x\in I[u,y]$. Now
  \AX{b3} implies $v\in I[u,x]$.  If $v=y$ then $v\in I[u,w]$ and $x\in
  I[v,w]$ then \AX{b3} implies $v\in I[u,x]$.
  (iii) $w\in\{x,y\}$: By assumption $w\ne x$.  If $w=y$ we have $x\in
  I[w,w]$, a contradiction.
  (iv) If $x=y$ we have $v\in I[u,w]\cap I[u,x]$ and hence the implication
  is trivial.
\end{proof}
It therefore sufficies to consider \AX{Sm} for any five pairwise distinct
vertices. Relaxing in \AX{Sm} the precondition that $uv$ and $wx$ are edges
yields the formally stronger axiom
\begin{itemize}
\item[\AX{Sm*}] If $v\in I[u,w]\cap I[u,y]$, and $x\in I[w,y]$, then $v\in
  I[u,x]$.
\end{itemize}
It is not difficult to show, however, that \AX{Sm*} is in fact equivalent
to smoothness for all connected graphs.
\begin{lemma}
  \label{lem:interval2}
  A connected graph $G$ is smooth if and only if any five vertices
  $u,v,w,x,y$ satisfy \AX{Sm*}.
\end{lemma}
\begin{proof}
  Since \AX{Sm*} holds in particular also if $uv\in E(G)$ and $wx\in E(G)$,
  \AX{Sm} holds trivially. For the converse, first observe that for any
  shortest path $w=x_0,x_1,\dots,x_{k-1},x_k=y$, we have $x_i\in
  I[x_{i-1},y]$ and $x_{i-1}x_i\in E(G)$. Thus, $v\in I(u,x_{i-1})$ implies,
  by \AX{Sm}, that $v\in I(u,x_i)$ for $1\le i<k$. Hence we can omit the
  condition $wx\in E(G)$. Now consider $v\in I[u,w]\cap I[u,y]$ and let
  $u_1,u_2,\dots,u_{k-1},u_k=v$ be a shortest path in $I[u,w]\cap I[u,y]$.
  Then $v\in I[u_{k-1},w]\cap I[u_{k-1},y]$ and thus by $x\in I[w,y]$
  implies (by the modification of \AX{Sm} justified above) that $v=u_k\in
  I[u_{k-1},x]$. Iterating this argument yields $u_j\in I[u_{j-1},x]$ and
  the betweeness property (b2), i.e., $I[u_j,x]\subseteq I[u_{j-1},x]$
  finally implies $v\in I[u,x]$.
\end{proof}

Clearly, smoothness is a graph property that is based solely on its
distance function $d$ as can be seen by the following alternative
formulation. In the light of Lemma~\ref{lem:interval2}, we have
\begin{observation}
  A connected graph $G$ is smooth if and only if the geodetic distances
  between any five vertices $u,v,w,x,y$ satisfy the following condition:
\begin{itemize}
\item[]
  If $d(u,w)=d(u,v)+d(v,w)$, $d(u,y)=d(u,v)+d(v,y)$, and
  $d(w,y)=d(w,x)+d(x,y)$, then $d(u,x)=d(u,v)+d(v,x)$.
\end{itemize}
\label{obs:dist} 
\end{observation}

We remark that smoothness of a graph can be checked efficiently. The
  geodetic distances between all pairs of vertices can be computed in
  $O(|V|^3)$ time and $O(|V|^2)$ space \cite{Dijkstra:59}. The distance
  condition in Obs.~\ref{obs:dist} can then be verified in constant time
  and space for any five vertices $u,v,w,x,y\in V$. Moreover, it suffices
  to consider the cases where $uv$ and $xw$ are edges. \AX{Sm} can thus be checked in $O(|E|^2|V|)$ time, which, in connected graphs, dominates
  the effort for pre-computing the distances. Memory consumption is
  dominated by storing the distance matrix.

The smoothness property is not hereditary because the deletion of vertices
(and edges) affects the distance function of the graph. Smoothness,
however, persists in a special class of (induced) subgraphs. Let
$W\subseteq V(G)$ and denote by $G[W]$ the subgraph of $G$ induced by $W$.
$G[W]$ is an \emph{isometric subgraph} of $G$ if it satisfies
$d_{G[W]}(u,v)=d_G(u,v)$. Note that any isometric subgraph is induced
because preservation of the distances in particular also preserves edges
and non-edges.
\begin{observation}
\label{ob:isometric}
  If $G$ is a smooth graph and $G[W]$ is an isometric subgraph of $G$,
  then $G[W]$ is smooth.
\end{observation}
In particular, therefore, connected induced subgraphs of smooth
distance-hereditary graphs are again smooth.

A \textit{retraction} of a connected graph $G$ is an idempotent
homomorphism $\phi:V(G)\to V(G)$, that is, $\phi$ satisfies
$d(\phi(x),\phi(y))\leq d(x,y)$ and $\phi(\phi(x))=\phi(x)$ for all $x,y\in
V(G)$. The subgraph $H\coloneqq G[\phi(V(G))]$ induced by the image of a
retraction is called a \textit{retract} of $G$. It is in particular, an
isometric subgraph of $G$, see \cite{hammack2011handbook}.
\begin{observation}
  If $G$ is a smooth graph and $H$ is a retract of $G$, then $H$ is smooth.
\end{observation} 

It is not difficult to verify that complete graphs, block graphs, and
cycles are smooth graphs, see \cite{nebesky2004properties}. On the other
hand, one easily checks that the five-vertex graphs $K_{2,3}$ and
$K_{1,1,3}$ are not smooth, see Figure~\ref{fig:non-smooth}.

\begin{figure}[h]
  \centering
  \begin{tikzpicture}[every node/.style={circle, draw, fill=white, inner sep=1.5pt}]
    \node (u) at (-6,0) {$u$};
    \node (v) at (-4,2) {$v$};
    \node (x) at (-4,-2) {$x$};
    \node (w) at (-2,0) {$w$};
    \node (y) at (-4,0) {$y$};
    \draw (u) -- (v) --(y) -- (x) -- (w) -- (v);
    \draw (x) -- (u);
    \node (u) at (0,0) {$u$};
    \node (v) at (2,2) {$v$};
    \node (x) at (2,-2) {$x$};
    \node (w) at (6,0) {$w$};
    \node (y) at (4,0) {$y$};
    \draw (u) -- (v) -- (y) -- (x) -- (w) -- (v);
    \draw (u) -- (x) -- (v);
  \end{tikzpicture}
  \caption{(a): $K_{2,3}$ (b): $K_{1,1,3}$
    The vertices $u,v,w,x,y$ are labeled to highlight the violation of
    \AX{Sm}.}  
  \label{fig:non-smooth}
\end{figure}
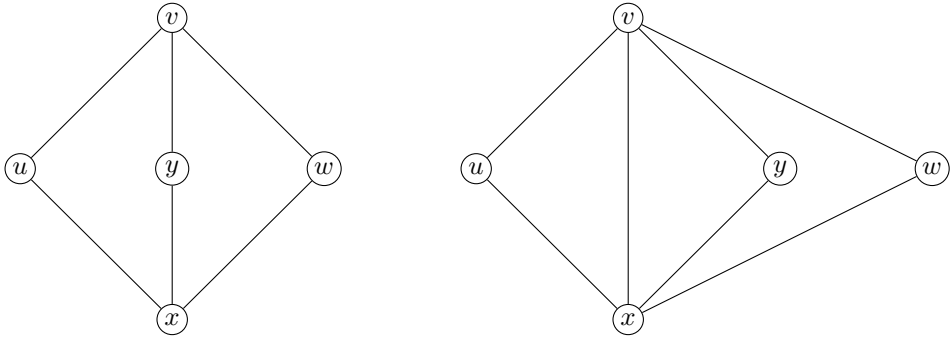

Graphs of diameter $2$ play a special role because of the following
simple fact:
\begin{observation}
  If $H$ is an induced subgraph of $G$ and $H$ has diameter at most $2$,
  then $H$ is an isometric subgraph of $G$.
\end{observation}
\begin{proof}
  For two distinct vertices we have $x,y\in V(H)\subseteq V(G)$ we have
  $d_G(x,y)=d_H(x,y)$ if $xy\in E(H)$ and $2\le d_G(x,y)\le d_H(x,y)\le 2$
  if $x$ and $y$ are not adjacent.
\end{proof}

 \begin{figure}[h]
   \centering
   \begin{tikzpicture}[every node/.style={circle, draw, fill=white, inner sep=1.5pt}]
     \node (u) at (-1.5,0) {$x$};
     \node (v) at (1.5,0) {};
     \node (x) at (0,1) {$w$};
     \node (w) at (-2,2) {$y$};
     \node (y) at (2,2) {$u$};
     \node (z) at (0,3) {$v$};
     \draw (u) -- (x) --(v) -- (y) -- (z) -- (w) -- (u);
     \draw (u) -- (v);
     \draw (x) -- (z);
   \end{tikzpicture}
   \caption{A $K_{2,3}$-and $K_{1,1,3}$-free graph that is not smooth. The
     five vertices violating \AX{Sm} are labeled.}
   \label{fig:brunograph1}
  \end{figure}
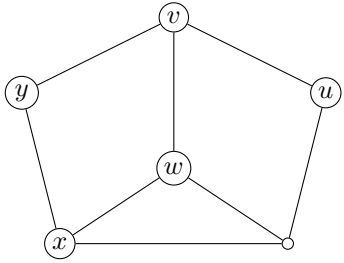
 
If $H$ is a graph, then $G$ is \emph{$H$-free} if it does not contain $H$
as an induced subgraph.
\begin{observation}
  If $H$ is a non-smooth graph with diameter at most two and $G$ is
  smooth, then $G$ is $H$-free.
\end{observation}
While this does not yield a characterization of smooth graphs, it does
restrict the class of smooth graphs to the graphs that are $K_{2,3}$-free
and $K_{1,1,3}$-free. This is not sufficient, however, to characterize
smooth graphs. For instance, the graph in Fig.~\ref{fig:brunograph1}
contains neither $K_{2,3}$ nor $K_{1,1,3}$ as an induced subgraph and is
still not smooth.

It is interesting to observe that the smoothness property appears as an
important ingredient in investigations of separation properties $S_4$ and
$S_3$ in abstract convexities, i.e., set systems $\mathcal{C}$ on a finite,
non-empty ground set $V$ satisfying (i) $V,\emptyset \in \mathcal{C}$ and
(ii) $\mathcal{C}$ is closed under intersections, see \cite{Vel:93}. For
two subsets $A$ and $B$ of the vertex set of a graph $G$, the \emph{shadow}
or \emph{extension} of $A$ with respect to $B$ is defined in
\cite{Chepoi:94} as the set $A/B \coloneqq \{x\in V:\, \langle B\cup
\{x\}\rangle \cap A\ne\emptyset\}$. Convexity of $A/B$ for all convex sets
$A,B$ is equivalent to the separation property $S_4$ and the convexity of
$u/B$ for any point $u$ and any convex set $B$ is equivalent to the
separation property $S_3$ in the abstract convexity and, in particular, in
the geodesic convexity of a graph, see \cite{Vel:93,Chepoi:94} and the
recent survey \cite{Chepoi:24}.  In the special case of $A$ and $B$
consisting of a single vertex, the set 
$$v/u \coloneqq \{x\in V:\, v\in\langle\{x,u\}\rangle\}$$ 
is the \emph{point-shadow} of $v$ with respect to $u$. The sets
\begin{equation} 
  U(v,u)\coloneqq\{x\in V:\, v\in I[x,u]\}
\end{equation}
are closely related to point-shadows. The following result shows that they
are also tightly linked to smoothness.

\begin{lemma}
  \label{lem:U(a,b)-convex} 
  A graph $G$ is a smooth graph if and only if the sets $U(v,u)$ are geodesically
  convex for any two adjacent vertices $u,v$ in $G$.
\end{lemma}
\begin{proof}
  Suppose that the sets $U(v,u)$ are geodesically convex for all adjacent
  pairs of vertices $u,v$.  Consider a graph $G$ and two adjacent vertices
  $u,v$ in $G$. Let $w,y$ be two distinct vertices in $U(v,u)$.  That is,
  $v\in I[u,w]$ and $v\in I[u,y]$. Now let $x \in I[w,y]$ by a vertex
  adjacent to $w$. Since $U(v,u)$ is geodesically convex, we have that
  $I[w,y] \subseteq U(v,u)$ and thus $x\in U(v,u)$, proving that $G$ is
  smooth.
 
  For the converse, assume that $G$ is a smooth graph and there are two
  adjacent vertices $u,v$ such that the set $U(v,u)$ is not geodesically
  convex. Then there exists two vertices $x,y\in U(v,u)$ and a vertex $z\in
  I[x,y]$ such that $z\notin U(v,u)$. That is, there exists a shortest
  $x,y$-path containing $z$, say $P= x,x_1\ldots x_i,x_{i+1},\ldots
  z,\ldots, y_{j-1},y_j\ldots,y$.  Since $z\in I[x,y]$, we
  can assume without loss of generality that the vertices $x_i,x_{i+1},y_{j-1},y_j$ in $P$ are such
  that $x_i,y_j\in U(v,u)$, but the vertices of the subpath $x_{i+1},\ldots
  z,\ldots, y_{j-1}$ of $P$ are not contained in $U(v,u)$.  Since $x,y\in
  U(v,u)$, we have $v\in I[u,x] \cap I[u,y]$. Since $G$ is smooth, $x_1\in
  U(v,u)$. Applying the smoothness property of $G$ iteratively to
  $u,v,y,x_1,\ldots x_i$ it follows that $x_2,\ldots x_i \in
  U(v,u)$. Applying again the smoothness property of $G$ to $u,v,y,x_i$, we
  obtain $x_{i+1}\in U(v,u)$, a contradiction.  Hence $I[x,y] \subseteq
  U(v,u)$, and $U(v,u)$ is convex.
\end{proof}

We observe that if $u$ and $v$ are adjacent vertices in $G$, then
\begin{equation} 
  U(v,u)=W_{vu}\coloneqq \{x\in V(G):\, d(x,v)<d(x,u)\}.
\end{equation}
A \emph{partial cube} is an isometric subgraph of a hypercube. As proved by
Djokovi{\'c}~\cite{Djokovic:73}, a connected bipartite graph $G$ is a
partial cube if and only if for every edge $uv\in E(G)$ the sets $W_{uv}$
and $W_{vu}$ are geodesically convex. By using
Lemma~\ref{lem:U(a,b)-convex}, we in turn infer the following
characterization of smooth graphs within the class of connected bipartite
graphs.
\begin{proposition} 
\label{p:partialcubes}
  A connected bipartite graph is smooth if and only if it is a partial
  cube.
\end{proposition}

It is already mentioned above that the geodesic convexity of the sets $A/B$
for all convex sets $A,B$ is equivalent to the separation property $S_4$,
which is equivalent to the well-known \textit{Pasch property} of a graph
$G$, which says that, \emph{for any five vertices $p,a,b, a',b'$ in $G$
satisfying, $a'\in I(p,a), b'\in I(p,b)$, the set $I(a',b)\cap I(b',a)$ is
non-empty} \cite{Vel:93,Chepoi:94}. We say that a graph has the Pash
property if its interval function is satisfying the Pasch property.
Proposition~4.10 of \cite{Vel:93} implies that the intervals $I[u,v]$ are
convex for every pair of vertices $u,v$ in a graph $G$ with the Pasch
property. In particular, therefore, $v/u=U(v,u)$.
Lemma~\ref{lem:U(a,b)-convex} now implies:
\begin{corollary}
  If $G$ is a graph with the Pasch property, then $G$ is smooth.
\end{corollary}

More generally, if $I$ satisfies the monotonicity property \AX{m}, which
states that $x,y\in I[u,v]$ implies $I[x,y]\subseteq I[u,v]$, viz, every
interval is convex, then we again have $v/u=U(v,u)$. This lets us conclude:
\begin{corollary}
  If the interval function of $G$ is monotone, then $G$ is smooth if and
  only if the point-shadows of all pairs of adjacent vertices are
  geodesically convex.
\end{corollary}

\begin{figure}[h]
  \centering
  \begin{tikzpicture}[every node/.style={circle, draw, fill=white, inner sep=1.5pt}]
    \node (u) at (-1,1.5) {$u$};
    \node (v) at (1,1.5)  {$v$};
    \node (x) at (1,3.5)  {$x$};
    \node (y) at (1,-0.5) {$y$};
    \node (w) at (3,1.5)  {$w$};
    \draw (u) -- (x) -- (w) -- (y) -- (u);
    \draw (u) -- (v) -- (x);
    \draw (v) -- (y);
   \end{tikzpicture}
  \caption{The graph $W_4^-$, the ``4-wheel minus a spoke'', is smooth.
    However, the point-shadow $v/u=\{v,w\}$ is not convex.}
  \label{fig:W4-}
\end{figure}
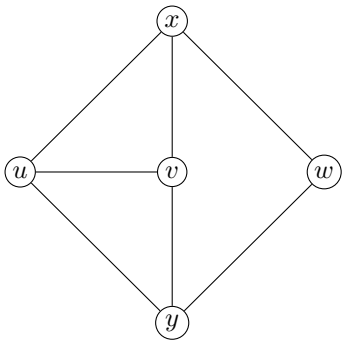

In general, smoothness does not imply convexity of point-shadows of
adjacent vertices. A counterexample is the graph $W_4^-$, the wheel of size
$4$ with a missing spoke, shown in Fig.~\ref{fig:W4-}. We shall see below,
however, that convexity of point-shadows is sufficient to ensure smoothness
and hence the ``smoothness property'' is a strictly weaker property than
the property of ``the convexity of point-shadow sets''.
  
For a subset $S\subseteq V(G)$, the convex hull $\langle S
\rangle$ can be determined iteratively starting from
$I^0[S]\coloneqq S$ by setting
\begin{equation}
  I^k[S] \coloneqq \bigcup_{x,y\in I^{k-1}[S] }I[x,y] \qquad\text{for }k>0 
\end{equation}
Then $\langle S\rangle =I^{k-1}[S] = I^k[S]$ for some $k>0$. The
  smallest such value of $k$ is known as the geodetic iteration number of
  $S$ \cite{Moscarini:20}.
\begin{lemma}
  If $G$ is a graph such that $v/u$ is geodesically convex for any two
  adjacent vertices $u,v$ in $G$, then $G$ is a smooth graph.
\end{lemma}
\begin{proof}
Assume that the sets $v/u$ are geodesically convex for any two
  adjacent vertices $u,v$. Let $w,y$ be two distinct vertices in $v/u$ such
  that $v\in I[u,w]\cap I[u,y]$. Now let $x\in I[w,y]$ be adjacent to
  $w$. Since $v/u$ is geodesically convex, we have $I[w,y]\subseteq v/u$
  and consequently $x\in v/u$. Also, $v/u = \langle v/u \rangle$. It
  remains to show that $v\in I[u,x]$.

Suppose, for contradiction, that $v\notin I[u,x]$. There exists a positive integer $k$ such that $\langle\{u,x\}\rangle = I^{k}[\{u,x\}] =I^{k-1}[\{u,x\}]$. That is, there exist
  vertex pairs, $(a_1,b_1)$, $(a_2,b_2)$, \ldots, $(a_k,b_k)$ such that
  $a_1,b_1\in I[u,x]$, $a_2,b_2\in I[a_1,b_1]$, \ldots, $a_k,b_k\in
  I[a_{k-1}, b_{k-1}]$ such that $v\in I[a_k,b_k]$. That is, there exist
  two distinct shortest $u,x$-paths, say $P_1$ and $P_2$, such that $P_1$
  contains $a_1$, $P_2$ contains $b_1$. Moreover, there is a shortest
    $a_k,b_k$-path containing $v$. Now consider the $v,a_1$-shortest path,
  say $Q_1$. Let $P_{a_1}$ be the $u,a_1$-subpath of $P_1$. Denoting
    by $\ell(P)$ the length a path $P$, we claim that
  $\ell(Q_1)=\ell(P_{a_1})$. Since $uv$ is an edge, it is clear that,
  $\ell(P_{a_1})$ is either $\ell(Q_1)$ or $\ell(Q_1)+1$. If
  $\ell(P_{a_1})=\ell(Q_1)+1$, then $u$ lies in a shortest $v,a_1$-path,
  and the shortest $v,a_1$-path through $u$ can clearly be extended to a shortest
  $v,x$-path through $P_1$.
    It follows that $u,x\in \langle
    \{v,x\}\rangle \subseteq \langle v/u \rangle$ because $v,x\in
    v/u$. Moreover, $v/u = \langle v/u \rangle$ implies $u\in v/u$, which
    is not possible according to the definition of $v/u$. Therefore we
    have $\ell(Q_1)=\ell(P_{a_1})$. Now, the shortest $v,a_1$-path $Q_1$
   can be extended to a shortest $v,x$-path, since
  otherwise, if there is a shorter $v,x$-path, then $v$ lies on a shortest
  $u,x$-path, which contradicts our supposition.  Since $a_1\in \langle
  \{v,x\}\rangle \subseteq \langle v/u \rangle$, we have $a_1\in \langle
  v/u \rangle$. Therefore $a_1\in v/u$.
  \\ If $v\in I[u,a_1]$, then the shortest $u,a_1$-path through $v$ can be
  extended to a shortest $u,x$-path through $v$, using the same arguments
  as in the above paragraph, and we obtain a contradiction to our
  supposition. Therefore $v\notin I[u,a_1]$. Now, replacing $x$ by $a_1$ and
  following the same arguments as in the previous situation, we obtain
  vertices $u_2, v_2$ lying in two distinct shortest $u,a_1$-paths such
  that $v\in \langle \{u_2,v_2\}\rangle$ and $u_2\in \langle v/u
  \rangle$. Therefore $u_2\in v/u$.
  \\ Now, replacing $a_1$ by $u_2$ and continuing the previous arguments,
  we obtain vertices $u_3$, \ldots, $u_j$, \ldots, $u_m$ lying on some
  shortest $u,a_1$-path (we can assume without loss of generality it is $P_{a_1}$) such that
  $u_3, \ldots u_j \ldots \in v/u$. Since there are finitely many vertices in $P_{a_1}$, we finally reach to the vertex $u_m$ in $P_{a_1}$ and hence in
  $P_1$, which is adjacent to the vertex $u$ and $v$ so that $u,v,u_m$ form
  a triangle and $u_m\in v/u$, which is clearly not possible, since $uu_m$ is an
  edge, yielding the final contradiction. We therefore conclude that
    $v\in I[u,x]$.
\end{proof}

\section{Product Graphs} 

\subsection{Cartesian Product}

The Cartesian product is one of the four standard graph products, and the
most important one in metric graph theory;
see~\cite{hammack2011handbook}. If $G$ and $H$ are graph, then the
\textit{Cartesian product} $G \square H$ of $G$ and $H$ has the vertex set
$V (G)\times V(H)$, and two vertices $(g,h)$ and $(g_0,h_0)$ are adjacent
if and only if $g=g_0$ and $hh_0 \in E(H)$, or $gg_0 \in E(G)$ and
$h=h_0$. The subgraphs of $G\square H$ induced by all vertices with fixed
$h$ or fixed $g$, called the layers of the product, are isomorphic to $G$
or $H$, respectively.  As a consequence, one can express the distance
$d_{G\square H}((g,h),(g_0,h_0))$ between two vertices in the product
graphs in terms of the distances $d_G(g,g_0)$ and $d_H(h,h_0)$ in the
factors:
\begin{equation}
  \label{eq:distance}
    d_{G \square H} ((g,h), (g_0,h_0 )) = d_G (g,g_0 ) + d_H (h,h_0).
\end{equation}
Equation~(\ref{eq:distance}) immediately implies that every layer is an
isometric subgraph of the Cartesian product $G\square H$. In addition, the
following relationship between the intervals in the two factors and their
product can be easily seen by using~\eqref{eq:distance} (see
e.g.\ \cite{mulder1980interval}).
\begin{lemma}
  \label{interval-product}
  If $G\square H$ is the Cartesian product of two (connected) graphs $G$
  and $H$ and $(g,h)$ and $(g_0, h_0)$ are its vertices, then
  \begin{equation}
    I_{G\square H}[(g, h), (g_0, h_0)] = I_G[g, g_0] \times I_H [h,
      h_0] = I_{G\square H} [(g_0, h), (g, h_0)]
    \label{eq:interval-product}
  \end{equation}
\end{lemma}

\begin{theorem}\label{thm:smoothproduct}
If $G$ and $H$ are connected graphs, then $G$ and $H$ are smooth graphs if
and only if $G\square H$ is a smooth graph.
\end{theorem}
\begin{proof}
  Let $G$ and $H$ be smooth graphs. Denote by $S$ the step system of
  $G\square H$ and consider vertices $(u_1,u_2)$, $(v_1,v_2)$, $(y_1,y_2)$,
  $(w_1,w_2)$, and $(x_1,x_2)$ in $V(G\square H)$ such that $(v_1,v_2)\in
  S((u_1,u_2),(w_1,w_2))$, $(v_1,v_2)\in S((u_1,u_2),(y_1,y_2))$ and
  $(x_1,x_2)\in S((w_1,w_2),(y_1,y_2))$.  Thus, we have
  $(u_1,u_2)(v_1,v_2)\in E(G\square H)$, and $(x_1,x_2)(w_1,w_2)\in
  E(G\square H)$, and hence either $u_1=v_1$ and $u_2v_2\in E(H)$ or
  $u_1v_1\in E(G)$ and $u_2=v_2$. Similarly, we have either $x_1=w_1$ and
  $x_2w_2\in E(H)$ or $x_1w_1\in E(G)$ or $x_2=w_2$.
  \par\noindent\textbf{Case 1:} $u_1=v_1$ and $u_2v_2\in E(H)$.  Now, we
  derive that $v_2\in S(u_2,w_2)$ and $v_2\in S(u_2,y_2)$ in $H$.  First,
  assume that $y_2\ne w_2$.  If $x_2w_2\in E(H)$, then $x_2\in
  S(w_2,y_2)$. Since $H$ is smooth, we have $v_2\in S(u_2,x_2)$. By
  Lemma~\ref{interval-product}, we infer that $(v_1,v_2)$ lies in a
  shortest path between $(u_1,u_2)$ and $(x_1,x_2)$, and hence
  $(v_1,v_2)\in S((u_1,u_2),(x_1,x_2))$. If $x_2=w_2$, then knowing $v_2\in
  S(u_2,w_2)$, which is equivalent to $v_2\in S(u_2,x_2)$, yields the same
  conclusion.  Second, assume that $y_2=w_2$. This implies $x_2=w_2$, and
  the same argument as above yields $(v_1,v_2)\in S((u_1,u_2),(x_1,x_2))$.
  \par\noindent\textbf{Case 2:} $u_1v_1\in E(G)$ and $u_2=v_2$. A similar
  argument can be used, this time based on the smoothness of $G$, to arrive
  at the same conclusion. Therefore $G\square H$ is smooth.
		
  Conversely, let $G\square H$ be a smooth graph. Since $G$ and $H$ appear
  as layers in $G\square H$, they are isometric subgraphs in $G\square H$,
  and, by Observation~\ref{ob:isometric}, they are smooth graphs.
\end{proof}

Since the edge $K_2$ is trivially smooth, hypercubes, i.e., $n$-fold
Cartesian products of edges are smooth. More generally, since complete
graphs are smooth, Hamming graphs, i.e., the Cartesian products of cliques
are smooth. This observation allows us to establish smoothness of many
important graph classes.

\emph{Partial Hamming graphs}~\cite{Bresar:01} are isometric subgraphs of
Hamming graphs, and thus a generalization of partial cubes. We immediately observe that partial Hamming graphs are smooth. This in turn implies that all members of
the hierarchy of graph classes discussed in~\cite{Imrich:98}, which in
particular contains the median graphs, are smooth. Smoothness of median
graphs was proved directly in~\cite[Proposition 4]{nebesky2004properties}.
Quasi-median graphs are a subclass of partial Hamming graphs, hence they
are smooth graphs. Alternatively, quasi-median graphs are precisely the
retracts of Hamming graphs~\cite{Chung:89,wilkeit1992retracts}, which gives
the same conclusion. In addition, their superclass of quasi-semimedian
graphs are also partial Hamming graphs (cf.~\cite{Bresar:03}), and hence
smooth.

The step system of connected graph $G$ satisfies axiom \AX{Sm} and an
additional axiom \AX{BP} [$v\in S(u,v)$ implies $v\in S(u,x)$ or $u\in
  S(v,x)$] if and only $G$ is a partial cube \cite{Changat:26a}. Step
systems satisfying \AX{BP}, in turn, are exactly connected
bipartite graphs. These results give an alternative argument that bipartite connected graphs are smooth if and only if they are partial cubes (see Proposition~\ref{p:partialcubes}).

\subsection{Strong Product}

The \textit{strong product} $G \boxtimes H$ of $G$ and $H$ is a graph with
vertex set $V (G)\times V(H)$, where $(g, h)(g',h')\in E(G\boxtimes H)$ if
and only if $gg'\in E(G)$ and $h=h$, or $g=g'$ and $hh'\in E(H)$, or
$gg'\in E(G)$ and $hh'\in E(H)$; see~\cite{hammack2011handbook}. For the
distances in the strong product we have
\begin{equation}
  \label{eq:distance-strong}
  d_{G \boxtimes H}((g,h),(g_0,h_0)) = \max\{d_G(g,g_0),d_H(h,h_0)\}
\end{equation}
Again, the subgraphs of $G\boxtimes H$ induced by the vertices with a fixed
first or second coordinate are isomorphic to $G$ and $H$, respectively.
For any two vertices in a $H$-layer, therefore we have $d_{G \boxtimes
  H}((g,h),(g_0,h_0)) = \max\{d_G(g,g_0),0\} = d_G(g,g_0)$, and for any two
vertices in a $G$-layer, we have $d_{G \boxtimes H}((g,h),(g_0,h_0)) =
\max\{d_G(g,g_0),0\} = d_G(g,g_0)$, and thus both $G$ and $H$ are
(isomorphic to) isometric subgraphs of $G\boxtimes H$. 

We will need the following auxiliary result.
\begin{lemma}
  \label{lem:pathprojection}
If $P:(g_0,h_0),(g_1,h_1),\dots,(g_{n-1},h_{n-1}),(g_n,h_0)$ is a
shortest $(g_0,h_0),(g_n,h_0)$-path in $G \boxtimes H$, then
the projection $p_G(P):g_0,g_1,\dots, g_{n-1},g_n$ is a
shortest $g_0,g_n$-path.
\end{lemma}
\begin{proof}
  First note that $d_{G\boxtimes H}((g_{i-1},h_{i-1}),(g_{i},h_{i}))=1$ and
  thus Eq.~\eqref{eq:distance-strong} implies $d_G(g_i,g_{i-1})\le 1$.
  Therefore, $p_G(P)$ is a trail in $G$ with possibly duplicate vertices an
  hence contains a $g_0,g_n$-path $P'$ of length $\ell(P')\le
  \ell(P)$. Moreover, $P'$ has an isomorphic copy $P^*$ in the $h_0$-layer
  of $G\boxtimes H$. In particular $P^*$ is therefore a
  $(g_0,h_0),(g_n,h_0)$-path of length $\ell(P^*)=\ell(P')$. Thus we have
  $d_{G\boxtimes H}((g_0,h_0),(g_n,h_0))=n=\ell(P)\le
  \ell(P^*)=\ell(P')\le\ell(P)$, which yields
  $\ell(P^*)=\ell(P')=\ell(P)=n$. Therefore, $P^*$ is an alternative
  shortest $(g_0,h_0),(g_n,h_0)$-path.  Finally,
  Eq.~\eqref{eq:distance-strong} implies $d_{G \boxtimes
    H}((g_0,h_0),(g_n,h_0)) =
  \max\{d_G(g_0,g_n),d_H(h_0,h_0)\}=d_G(g_0,g_n)=n$. Since the trail
  $p_G(P)$ has length $n$ by construction, it cannot contain any duplicate
  vertices, and hence $p_G(P)=P'$ is a shortest $g_0,g_n$-path in $G$.
\end{proof}

This result has an immediate translation to the step systems:
\begin{corollary}
  \label{cor:step}
  If $S$ is the step system of $G\boxtimes H$ and $S_G$ the step
  system of $G$, then $(x_1,x_2)\in S((u_1,h)(v_1,h))$ implies $(x_1,h)\in
  S((u_1,h)(v_1,h))$ and $x_1\in S_G(u_1,v_1)$.
\end{corollary}
Of course an analogous implication holds for the step system $S_H$ in $H$
if $(x_1,x_2)\in S((g,u_2),(g,v_2))$.

\begin{theorem}
  \label{thm:strong}
  If $G$ and $H$ are connected graphs, then $G\boxtimes H$ is smooth if and
  only if $G$ and $H$ are smooth graphs.
\end{theorem}
\begin{proof}
  Let $G$ and $H$ be smooth graphs. Denote by $S$ the step system of
  $G\boxtimes H$ and consider a set $X\coloneqq \{ (u_1,u_2), (v_1,v_2),
  (y_1,y_2), (w_1,w_2), (x_1,x_2)\}\in V(G\boxtimes H)$ of five vertices
  such that $(v_1,v_2)\in S((u_1,u_2),(w_1,w_2))$, $(v_1,v_2)\in
  S((u_1,u_2),(y_1,y_2))$ and $(x_1,x_2)\in S((w_1,w_2),(y_1,y_2))$.
  By assumption we have $(u_1,u_2)(v_1,v_2)\in E(G\boxtimes H)$ and
  $(x_1,x_2)(w_1,w_2)\in E(G\boxtimes H)$.
  If $u_1=v_1$ and $u_2v_2\in E(H)$ or $u_1v_1\in E(G)$ and $u_2=v_2$, the
  same arguments as in the proof of Theorem~\ref{thm:smoothproduct} implies
  that the five vertices in $X$ satisfy \AX{Sm}.  In the
  following we therefore assume $u_1v_1\in E(G)$ and $u_2v_2\in E(H)$.
  We distinguish three cases: 
  \par\noindent\textbf{Case 1:}
  If $w_1x_1\in E(G)$ and $w_2=x_2$, then consider a shortest $(u_1,u_2),
  (w_1,w_2)$-path $P$ that contains $(v_1,v_2)$. Follow $P$ starting with
  $(u_1,u_2),(v_1,v_2)$ and stop right before reaching the $G$-layer that
  contains $(w_1,w_2)$, and let $(z_1,z_2)$ be the corresponding last
  vertex. Then either $(z_1,z_2)(x_1,x_2)$ is an edge, and the path
  finishes as a shortest $(u_1,u_2), (x_1,x_2)$-path or we follow $P$
  further until reaching $(w_1,w_2)$. Again, we can either reach
  $(x_1,x_2)$ a step earlier, or we continue with the edge
  $(w_1,w_2)(x_1,x_2)$, in either case obtaining a shortest $(u_1,u_2),
  (x_1,x_2)$-path that contains $(v_1,v_2)$.
  \par\noindent\textbf{Case 2:}
  If $w_2=x_2\in E(H)$ and $w_2x_2\in E(H)$ we can argue analogously.
  \par\noindent\textbf{Case 3:} If $w_1x_1\in E(G)$ and $w_2x_2\in E(H)$ we
  use $d_{G\boxtimes
    H}(u_1,u_2),(w_1,w_2)=\max\{d_G(u_1,w_1),d_H(u_2,w_2)\}$ and assume,
  without loss of generality, that the maximum is attained by
  $d_G(u_1,w_1)=\ell$ and thus the shortest $(u_1,u_2),(x_1,x_2)$-path
  through $(v_1,v_2)$ is also of length $\ell$.  Thus $(v_1,v_2)\in
  S((u_1,u_2),(x_1,x_2))$ and hence $X$ satisfies \AX{Sm}.  Therefore
  $G\boxtimes H$ is smooth.
  For the converse, it suffices to observe that $G$ and $H$ are isometric
  subgraphs of $G\boxtimes H$, and thus smooth whenever $G\boxtimes H$ is
  smooth. 
\end{proof}

The Helly graphs, also known as absolute retracts \cite{Hell:87}, are the
retracts of strong products of paths \cite{Nowakowski:83,Jawhari:86}, and
therefore smooth graphs.

In \cite[Theorem 2.2]{Imrich:92} Imrich and Klav{\v{z}}ar showed that retracts
of strong products are strong products of isometric subgraphs of the
factors. As an immediate consequence we have
\begin{corollary}
  If $G$ is smooth and a retract of $H_1\boxtimes H_2$ then $G=G_1\boxtimes
  G_2$, where both $G_1$ and $G_2$ are isometric subgraphs of $H_1$ and
  $H_2$, respectively.
\end{corollary}

\subsection{Lexicographic Product}

The \textit{lexicographic product} of graphs $G$ and $H$ is the graph $G
\circ H$ with vertex set $V (G) \times V (H)$.  We have
$(g,h)(g_0,h_0)\in E(G\circ H)$ if $gg_0\in E(G)$, or $g=g_0$ and
$hh_0 \in E(H)$. In contrast to the Cartesian and strong products, the
lexicographic product is not commutative. Indeed the two factors contribute
very differently to the distance in the product:
\begin{equation} \label{eq:distance-lexico}
  d_{G\circ H}((g,h),(g_0,h_0))=
  \begin{cases}
    d_G(g,g_0)         & \text{if } g\neq g_0\\
    \min(2,d_H(h,h_0)) & \text{if } g = g_0              
  \end{cases}
\end{equation}

\begin{figure}[h!]
  \centering
  \begin{tikzpicture}[every node/.style={circle, draw, fill=white, inner sep=1.5pt}]
    \node (u) at (1,2) {$u$};
    \node (v) at (1,0) {$v$};
    \node (w) at (1,-2) {$w$};
    \node (a) at (2,3) {$a$};
    \node (b) at (4,3) {$b$}; 
    \node (c) at (6,3) {$c$};
    \draw (u) -- (v) --(w); 
    \draw (a) -- (b) -- (c);
    \node (ua) at (2,2) {$~$};
    \node (ub) at (4,2) {$~$};
    \node (uc) at (6,2) {$~$};
    \draw (ua) -- (ub) -- (uc);
    \node (va) at (2,0) {$~$};
    \node (vb) at (4,0) {$~$};
    \node (vc) at (6,0) {$~$};
    \draw (va) -- (vb) -- (vc);
    \node (wa) at (2,-2) {$~$};
    \node (wb) at (4,-2) {$~$};
    \node (wc) at (6,-2) {$~$};
    \draw (wa) -- (wb) -- (wc);
    \draw (ua) -- (va);
    \draw (ua) -- (vb);
    \draw (ua) -- (vc);	
    \draw (ub) -- (va);
    \draw (ub) -- (vb);
    \draw (ub) -- (vc);	
    \draw[magenta, very thick] (uc) -- (va);
    \draw (uc) -- (vb);
    \draw [magenta, very thick](uc) -- (vc);	
    \draw[magenta, very thick] (va) -- (wa);
    \draw (va) -- (wb);
    \draw[magenta, very thick] (va) -- (wc);	
    \draw (vb) -- (wa);
    \draw (vb) -- (wb);
    \draw (vb) -- (wc);	
    \draw[magenta, very thick] (vc) -- (wa);
    \draw (vc) -- (wb);
    \draw[magenta, very thick] (vc) -- (wc);	
  \end{tikzpicture}
  \caption{The lexicographic product $P_3\circ P_3$ contains a
    $K_{2,3}$ as an induced subgraph.}
  \label{fig:not-smooth-lexico}
\end{figure}
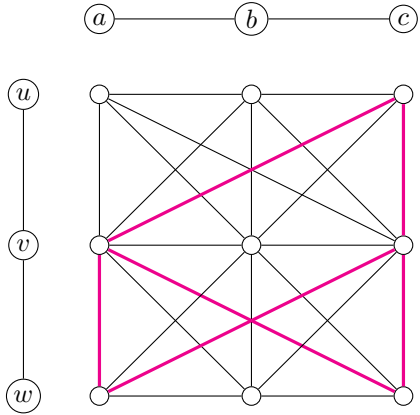

In particular, we have $d_{G\circ H}((g,h),(g,h_0))=1$ if $hh_0\in E(H)$
and $d_{G\circ H}((g,h),(g,h_0))=1$ for two distinct non-adjacent vertices
$h,h_0\in V(H)$. As in the Cartesian and strong product, layers form
induced subgraphs. However, the $H$-layers are not isometric subgraphs.

Suppose both $G$ and $H$ contain induced paths $P_3$ of length two,
$P_G=(u,v,w)$ in $G$ and $P_H=(a,b,c)$. Then $P_G\circ P_H$ is by
definition of the lexicographic product an induced subgraph of $G\circ
H$. By Equ.(\ref{eq:distance-lexico}) $P_G\circ P_H$ has diameter $2$ and
thus an isometric subgraph of $G\circ H$, see
Fig.~\ref{fig:not-smooth-lexico}.  Since $P_3\circ P_3$ contains a
$K_{2,3}$ as an induced subgraph it is not smooth. Therefore lexicographic
products can be smooth only if one the factors has diameter $1$ and thus is
a complete graph (including a single edge). 

\section{Gated Amalgams}

A subset $W\subseteq V(G)$ is called \textit{gated} if for every vertex $u$
in $G$ there exists a vertex $u_W$ in $W$ such that $u_W$ lies on all
shortest $u,v$-paths for all $v\in W$
\cite{Goldman:70,dress1987gated}. Let $H=G[W]$ be the subgraph of $G$
induced by a gated set $W$.  If $W$ is a gated set in $G$ then
\begin{equation}
  d_{H}(u_W,v_W) \le d_G(u,v)
\end{equation}
If $u\in W$, then $u_W=u$, i.e., in this case $u$ its its own gate for $W$,
and hence $d_{G[W]}(u,v)=d_G(u,v)$ for $u,v\in W$.  The gate function
mapping $u$ to the corresponding gate $u_U$ of a gated subset $U$ in $G$ is
a weak retraction and, in particular, any subgraph $G[U]$ of $G$ induced by
a gated set $H$ is an isometric subgraph of $G$. Thus we have
\begin{observation}
  If $G$ is smooth and $W$ is a gated set in $G$ then $G[W]$ is smooth. 
\end{observation}

A graph $G$ is the \emph{gated amalgam} of $G_1$ and $G_2$ if $G$ contains
two gated sets $W_1,W_2\subset V(G)$ such that $W_1\cap W_2\ne\emptyset$,
$W_1\cup W_2=V(G)$, $G_1\simeq G[W_1]$ and $G_2\simeq G[W_2]$. In the
following we will simply identify $G_i$ and $G[W_i]$. 

If $W_1\cap W_2=\{x\}$ and $x$ is a cut vertex, then both $W_1$ and $W_2$
are gated, since $x=u_{W_1}$ for all $u\in W_2$ and $x=u_{W_2}$ for all
$u\in W_1$. Thus gluing two graph together a single (cut) vertex is a
special case of a gated amalgamation.
Proposition~3 in~\cite{nebesky2004properties} states that a 
graph is smooth if and only if each of its blocks is smooth. This result
can be stated equivalently as follows:
\begin{corollary}
  If $G$ is a graph obtained by gluing together two subgraphs $G_1$ and
  $G_2$ at a single cut vertex, then $G$ is smooth if and only if $G_1$ and
  $G_2$ are smooth.
\end{corollary}
This begs the question whether this result remains true for gated
amalgams in general. The following theorem gives an affirmative answer.

\begin{theorem}\label{gated-amalgam-smooth}
The gated amalgam of two smooth graphs is a smooth graph. 
\end{theorem}
\begin{proof}
  Let $G$ be a gated amalgam of two smooth graphs $G_1=G[W_1]$ and
  $G_2=G[W_2]$ and assume that $G$ is not smooth. Thus there is a set
  $X\coloneqq \{u,v,w,x,y\}\subset V(G)$ of five vertices that serve as a
  counterexample, i.e., that violates \AX{Sm}.  Since $W_1$ and $W_2$ are
  gated, and thus $G[W_1]$ and $G[W_2]$ are isometric subgraphs of $G$, we
  conclude that the counterexample $X$ cannot by contained entirely in
  $W_1$ or in $W_2$.  Recall that the counterexample $X$ has the following
  structure: (i) $uv$ and $wx$ are edges, there is a shortest $u,w$-path
  and $u,y$-path passing through $v$, and a shortest $w,y$-path through
  $x$, but there is no shortest $u,x$-path through $v$.

  Consider $X\cap (W_1\setminus W_2)\ne\emptyset$.  We proceed case-by-case
  we show show that no such counterexample can exist.  First note that
    we may assume that that $X\cap (W_2\setminus W_1)\ne\emptyset$, since
    otherwise $X\subseteq W_1$. Moreover, the absence edges of between
  $W_1\setminus W_2$ and $W_2\setminus W_1$ immediately exclude the
  following sitatuations: \par\noindent\textbf{(a)} $v\in W_1\setminus W_2$
  and $u\in W_2\setminus W_1$.  \par\noindent\textbf{(b)} $w\in
  W_1\setminus W_2$ and $x\in W_2\setminus W_1$ \par\noindent\textbf{(c)}
  If $p\in X\setminus\{v\}$ is the only vertex of $X$ in $W_1\setminus
  W_2$, then every shortest path from $p$ must also run through its gate
  $p_2\in W_2$, and thus we obtain an alternative counterexample
  $X'\coloneqq (X\setminus\{p\})\cup\{p'\}$. However, $X'\subseteq W_2$ and
  thus cannot be a counterexample. Together with \textbf{(a)}, this implies
  that at least two vertices are located in $W_1\setminus W_2$.
  
  Considering \textbf{(a)}, \textbf{(b)}, and \textbf{(c)}, and noting that,
  by symmetry, we may assume without losing generality
  that $u\in W_1$, the following cases remain: 

  \par\noindent\textbf{Case 1:}
  $u,v\in W_1\setminus W_2$ and $w,x,y\in
  W_2$.\\
  Denote by $z_u$ and $z_v$ be the gates of $u$ and $v$ in $G_2$.  Since
  $z_u$ and $z_v$ are the gates (and hence image of the retraction) of
  adjacent vertices we have either $z_u=z_v$ or $z_u$ and $z_v$ are
  adjacent. We assume, for contradiction that $z_u\ne z_v$.  Uniqueness of
  the gate implies that all shortest $u,w$-path and all shortest
  $u,y$-paths pass through the gate of $u$, i.e., $z_u$. Since $v\in
  S(u,w)$, there exists a shortest $u,w$-path $P$ passing through $v$, and
  thus $z_u\in P$ is the first vertex in $W_2$ that $P$
  hits. However, as the $v,w$-subpath of $P$ is a shortest $v,w$-path, the
  first vertex in $W_2$ hit by $P$ is $z_v$. Therefore $z_u=z_v$
  and every shortest $u,x$-path passes through $z_u$.  Thus $v$ lies
  between $u$ and $z_u$ along a shortest $u,w$-path, and hence also along a
  shortest $u,x$-path.
  \par\noindent\textbf{Case 2:}  $u,v,w,x\in W_1\setminus W_2$ and
  $y\in W_2$.
  \\
  Let $y'$ be the gate of $y$ in $G_1$.  Consider the vertices
  $u,v,w,x,y'$. All these vertices lies in $G_1$.  Since $y'$ is the gate
  of $y$ in $G_2$, $y'$ lies on a shortest path between $y$ and any vertex
  in $X\setminus\{y\}$.  Hence, there exists a shortest $u,y'$-path passing
  through $v$ and a shortest $w,y'$-path through $x$, since $uv$ and $wx$
  are edges. These vertices satisfies the pre-conditions of smoothness,
  namely, $v$ is on a $u,w$-shortest path and on a $u,y'$ shortest path and
  $x$ is on a shortest $w,y'$-path. Thus assuming that $v$ is not on a
  shortest $u,x$-path in $G$ implies that $G_1$ is not smooth,
  contradicting our assumptions.
  \par\noindent\textbf{Case 3:} $u,v,y\in W_1\setminus W_2$ and
  $w,x\in W_2$.
  \\
  Since $x$ lies on a shortest $w,y$-path, we infer in a similar way as in
  Case 1 that the gate of $x$ in $W_1$ is the same as the gate of $w$ in
  $W_1$; denote it by $z$. Consider a shortest $u,w$-path $P$ containing
  $v$. Clearly, $P$ contains also $z$, and in turn it contains $x$. Thus
  there is a subpath of $P$, which is a shortest $u,x$-path that contains
  $v$.
  Consider the gates $w'$ of $w$ and $x'$ of $x$ in $G_1$. Since $wx$ is an
  edge in $G_2$, $w'$ is distinct from $x'$ and thus $w'x'$ is be an edge
  in $G_1$ since the mapping of vertices to their gates is a retraction.
  Now consider $X'\coloneqq\{u,v,y,w',x'\}$. We have $X'\subseteq W_1$.
  Similar to Case 2 we see that the $X'$ satisfies the pre-conditions of
  \AX{Sm} but violate the implication, contradicting the assumption that
  $G_1$ is smooth.
  The final possibility is $x,w\in W_1\setminus W_2$, for which we have
  already ruled out all possibilities to form a counterexample: if $u,v\in
  W_1\setminus W_2$, we obtain a subcase of Case 2; if $u,v\in W_2\setminus
  W_1$, we obtain a subcase of Case 1; using \textbf{(a)} and \textbf{(c)},
  only $u,v\in W_1\cap W_2$ and thus $u,v,w,x\in W_1$.  Moreover, by
  \textbf{(c)}, $y\in W_2\setminus W_1$ is impossible and hence $y\in W_1$,
  and thus $X\subseteq W_1$.  Summarizing \textbf{(a)}, \textbf{(b)},
  \textbf{(c)} and the three Cases 1--3, there is no choice of $X$ that
  satisfies the precondition of \AX{Sm} but violates its implication.
\end{proof}

\section{Scaled Embeddings into Hypercubes}
\label{sec:l1}

\begin{definition}
  A connected graph $G$ with shortest path distance $d_G$ is
  \emph{embeddable with scale $\lambda\in\mathbb{N}$} into a graph $H$ with
  shortest path distance $d_H$ if there exists a mapping $\varphi: V(G)\to
  V(H)$ such that $d_{H}(\varphi(x),\varphi(y)) = \lambda d_G(x,y)$ for
  every pair of vertices $x,y\in V(G)$.
\end{definition}

\begin{theorem}\label{scale2}
  Let $G$ be a graph, and let $H$ be a smooth graph such that $G$ has a
  scale $2$ embedding into $H$. Then $G$ is also a smooth graph.
\end{theorem}
\begin{proof}
  As noted in Obs.~\ref{obs:dist}, the smoothness condition for $G$ and $H$
  can be expressed entirely in terms of the distances $d_G$ and $d_H$.
  Assume $u,v,w,x,y\in V(G)$ satisfy the precondition in
  Obs.~\ref{obs:dist}. Since by assumption there exists $\varphi:V(G)\to V(H)$
  such that $d_{H}(\varphi(u),\varphi(v))=2d_G(u,v)$, we have a shortest
  $\varphi(u),\varphi(w)$-path passing through $\varphi(v)$, a shortest
  $\varphi(u),\varphi(y)$-path passing through $\varphi(v)$ and a shortest
  $\varphi(w),\varphi(y)$-path passing through $\varphi(x)$.  Moreover, we
  have $d_H(\varphi(u),\varphi(v))= 2$ and $d_H(\varphi(w),\varphi(x))=2$.
  This situation in $H$ is depicted in Fig.~\ref{fig:placeholder}.  $G$ is
  smooth if these conditions imply $d_G(u,x)=d_G(u,v)+d_G(v,x)$, or
  equivalently $d_{H}(\varphi(u),\varphi(v))+ d_{H}(\varphi(v),\varphi(x))
  = d_{H}(\varphi(u),\varphi(x))$ in $H$.

  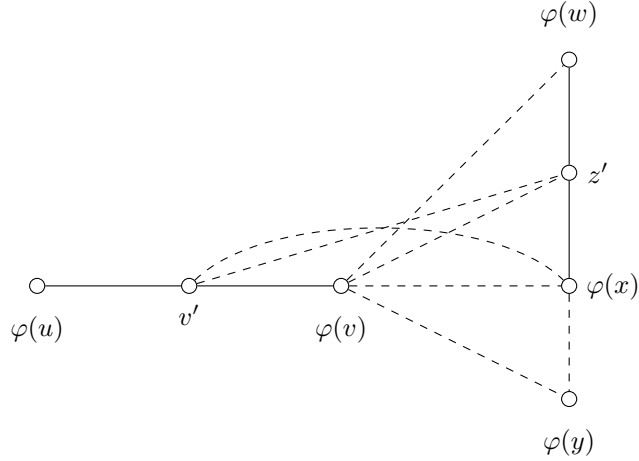
\begin{figure}[h]
    \centering
    \begin{tikzpicture}[every
        node/.style={circle, draw, fill=white, inner sep=1.5pt}]
      \node[draw,circle,inner sep=2pt,
        label=below:$\varphi(u)$] (v1) at (-4,0) {};
      \node[draw,circle,inner sep=2pt,
        label=below:$v'$] (v2) at (-2,0) {};
      \node[draw,circle,inner sep=2pt,
        label=below:$\varphi(v)$] (v3) at (0,0) {};
      \node[draw,circle,inner sep=2pt,
        label=above:$\varphi(w)$] (v4) at (3,3) {};
      
      \node[draw,circle,inner sep=2pt,
        label=right:$z'$] (v5) at (3,1.5) {};
      \node[draw,circle,inner sep=2pt,
        label=right:$\varphi(x)$] (v6) at (3,0) {};
      \node[draw,circle,inner sep=2pt,
        label=below:$\varphi(y)$] (v7) at (3,-1.5) {};
      \draw (v1) -- (v2);
      \draw (v2) -- (v3);
      \draw[dashed] (v3) -- (v4);
      \draw[dashed] (v3) -- (v5);
      \draw[dashed] (v3) -- (v6);
      \draw[dashed] (v3) -- (v7);
      \draw[dashed] (v2) -- (v5);
      \draw[dashed] (v2)  .. controls +(1,1) and +(-1,1) .. (v6);
      \draw (v4) -- (v5);
      \draw (v5) -- (v6);
      \draw[dashed] (v6) -- (v7);
    \end{tikzpicture}
    \caption{Mutual relationships of $\varphi(u)$, $\varphi(v)$,
      $\varphi(w)$, $\varphi(x)$, and $\varphi(y)$ in $H$ for $u,v,w,x,y\in
      V(G)$ forming the precondition for the smoothness condition of
      Obs.~\ref{obs:dist}.}
    \label{fig:placeholder}
  \end{figure}
  
  It therefore remains to show that the latter equality indeed holds. To
  this end, since $d_H(\varphi(u),\varphi(v))=2$, and $\varphi(v)$ is on
  both a shortest $\varphi(u),\varphi(w)$-path and a
  $\varphi(u),\varphi(y)$-path, there exists a vertex $v'$ adjacent to the
  vertices $\varphi(u)$ and $\varphi(v)$ and on both the shortest
  $\varphi(u),\varphi(w)$-path and a
  $\varphi(u),\varphi(y)$-path. Similarly, since
  $d_H(\varphi(w),\varphi(x))=2$, and $\varphi(x)$ is on a shortest
  $\varphi(w),\varphi(y)$-path, there exists a vertex $z'$ adjacent to
  $\varphi(w)$ and $\varphi(x)$ and on a shortest
  $\varphi(w),\varphi(y)$-path (see the Figure~\ref{fig:placeholder}).  Now
  consider the vertices $v',z'\in V(H)$ with $d_H(\varphi(u),v')=1$ and
  $d_H(\varphi(x),z')=1$.  Consider a shortest
  $\varphi(u),\varphi(w)$-path passing through $v'$ and $\varphi(v)$, a
  shortest $\varphi(u),\varphi(y)$-path passing through $v'$ and
  $\varphi(y)$ and a shortest $\varphi(w),\varphi(y)$-path passing
  through $z'$ and $\varphi(x)$.  Since $H$ is smooth, there is a shortest
  $\varphi(u),z'$-path passing through $v'$.

  Now consider a shortest $\varphi(u),z'$-path through $v'$, a shortest
  $\varphi(u),\varphi(y)$-path passing through $v'$ and $\varphi(v)$ and
  a shortest $z',\varphi(y)$-path passing through $\varphi(x)$. Invoking
  smoothness of $H$, there is a shortest $\varphi(u),\varphi(x)$-path $P$
  passing through $v'$.

  Considering the smoothness property for the five vertices $v'$,
  $\varphi(v)$, $\varphi(w)$, $z'$ and $\varphi(y)$, we conclude that there
  is a shortest $v',z'$ path passing through $\varphi(v)$. Applying the
  smoothness condition again to the vertices $v'$, $\varphi(v)$, $z'$,
  $\varphi(x)$, and $\varphi(y)$, we obtain a shortest $v',\varphi(x)$-path
  $P'$ passing through $\varphi(v)$. The path $P$ formed by the edge
  $v'\varphi(u)\in E(H)$ and $P'$ is a shortest $\varphi(u),\varphi(x)$-path through
  $\varphi(v)$, because $P$ is a shortest $\varphi(u),\varphi(x)$-path
  passing through $v'$. This implies the desired equality
  $d_{H}(\varphi(u),\varphi(v))+d_{H}(\varphi(v),\varphi(x))=
    d_{H}(\varphi(u),\varphi(x)))$.
\end{proof}

The half-cube graph $Q^n/2$ (see e.g.\ \cite{Imrich:98a}) is the square of
$Q^{n-1}$, i.e., the graph with $V(Q^n/2)=V(Q^{n-1})$ and $x,y\in E(Q^n/2)$
if and only if $1\le d_{Q^{n-1}}(x,y)\le 2$. The half-cube $Q^n/2$
therefore has an embedding with scale-$2$ into the hypercube $Q^n$. As an
immediate consequence of Theorem~\ref{scale2} and smoothness of $Q^n$, we
therefore have
\begin{corollary}
  \label{cor:halfcube}
  If $G$ is a half-cube graph then $G$ is smooth.
\end{corollary}
We note that the isometric subgraphs of half-cubes are the weakly median
graphs without an induced $K_{1,1,1,1,2}=K_6-e$
\cite[Thm.2]{bandelt2000decomposition}. Recall that a {\em cocktail-party graph} (or a {\em hyperoctahedron}) is a complete graph minus a perfect matching. 

\begin{lemma}
  Every cocktail-party graph is smooth.
  \label{lem:cocktail} 
\end{lemma}
\begin{proof}
  Let $G$ be a cocktail-party graph. Thus for every $u\in V(G)$ there is exactly one $u^*\in V(G)$ such that $u^*u\notin V(G)$. By Lemma~\ref{lem:5only} it suffices to
  consider the smoothness condition of Obs.~\ref{obs:dist} for five
  pairwise distinct vertices $u,v,w,x,y\in V(G)$. If the precondition of
  the statement in Obs.~\ref{obs:interval} is satisfied, then none of $uw$,
  $uy$, and $wy$ can be an edge, i.e., for each of $u$, $v$, and $w$ there
  are two non-adjacent vertices in $G$. However, in a cocktail party graph,
  there is exactly one non-adjacent vertex, namely the partner in the
  perfect matching.
\end{proof}

\begin{definition}
  A graph $G$ is a \emph{$\ell_1$-graph} if it has a \emph{scale-$\lambda$
  embedding} in a hypercube, i.e., if there an $n\ge 1$ and map $V(G)\to
  V(Q_n)$ such that $d_{Q_n}(\varphi(x),\varphi(y)) = \lambda d_G(x,y)$ for
  some fixed positive integer constant $\lambda$.
\end{definition}

Shpectorov \cite{Shpectorov:93} proved that a finite graph $G$ is an
$\ell_1$-graph if and only if it is an isometric subgraph of the Cartesian
product of cocktail party graphs (hyperoctahedra), complete graphs, and
half-cubes. This characterization yields the main result of this section:
\begin{theorem}
  Every $\ell_1$-graph is smooth.
  \label{thm:L1-smooth}
\end{theorem}
\begin{proof}
  By Cor.~\ref{cor:halfcube}
  and Lemma~\ref{lem:cocktail}, both cocktail party graphs and half-cubes
  are smooth. By Thm.~\ref{thm:smoothproduct} their Cartesian products are
  smooth and Obs.~\ref{ob:isometric}, isometric subgraphs of smooth graphs
  are smooth. Therefore every $\ell_1$-graph is smooth.
\end{proof}

The converse of Theorem~\ref{thm:L1-smooth} is not true, i.e., there are
smooth graphs that are not $\ell_1$. A small example is $W_4^-$;
see~\cite[Lemma 4.24]{chalopin2020weakly}. 

\section{Ptolemaic Graphs}

\emph{Ptolemaic graphs} were introduced by Kay and Chartrand in
\cite{kay1965characterization} as graphs in which the distances obey the
Ptolemy's inequality.  That is, the distances between any four vertices
$u,v,w,x\in V(G)$ satisfy $d(u, v)d(w,x) + d(u, x)d(v, w) \geq d(u, w)d(v,
x)$. Ptolemaic graphs are exactly the distance-hereditary chordal
graphs \cite{howorka1981characterization}, the 3-fan-free chordal graph
\cite{howorka1981characterization}, and the $C_4$-free distance-hereditary
graphs \cite{McKee:10}. Since $K_{1,1,3}$ is Ptolemaic, it is clear
that not all Ptolemaic graphs are smooth. As it turns out, however, there
are no other obstructions:

\begin{theorem}
  \label{smooth-Ptolemaic}
  If $G$ is a Ptolemaic graph, then $G$ is smooth if and only if $G$ is
  $K_{1,1,3}$-free.
\end{theorem}
\begin{proof}
  Since all smooth graphs are $K_{1,1,3}$-free, this is in particular also
  true for smooth Ptolemaic graphs. Now let $G$ be a $K_{1,1,3}$-free
  Ptolemaic graph. If $|V(G)|\le 4$, then $G$ is trivially smooth.  Assume
  that $G$ is a Ptolemaic graph with $|V(G)|\ge 5$ that is not smooth. Then
  there exists five distinct vertices $u,v,w,x,y$ such that $uv$ and $wx$
  are edges and there are shortest $u,w$- and $u,y$-path through $v$ and a
  shortest $w,y$-path through $x$, but no $u,x$-shortest path through
  $v$. Without loss of generality, may assume that the vertices $u,v,w,x,y$
  form a counterexample in which the distance $d(w,y)$ is the minimum. Let
  $P_{v,w}$, $P_{v,y}$ and $P_{wx,y}$ be shortest $v,w$-, $v,y$- and
  $w,y$-paths, where $P_{wx,y}$ runs through $x$.  Let $v'$ be the last
  common vertex between $P_{v,w}$ and $P_{v,y}$ as we traverse from $v$ and
  let $y_0$ be the first common vertex between $P_{wx,y}$ and $P_{v,y}$. If
  $y_0=v'$, then $x$ lies between $v'$ and $w$ along the shortest
  $u,w$-path $(u,v\dots v',\dots w)$, contradicting our assumption.  Thus
  $y_0\ne v'$.  If $y_0\ne y$, then $\{u,v,w,x,y_0\}$ is counterexample
  with $d(w,y_0)<d(w,y)$, violating the assumption $d(w,y)$ is
  minimal. Hence $y_0=y$.  Let $P_{v',y}$ denote the subpath of $P_{v,y}$
  and $P_{v',w}$ denote the subpath of $P_{v,w}$.
  \smallskip
  \par\noindent\emph{{\bfseries{Claim:}} The cycle $C\coloneqq P_{v',w}
  \cup P_{v',y} \cup P_{wx,y}$ is an induced cycle in $G$.}
  \par\noindent {\em Proof of Claim.} Suppose that $C$ is not induced. Then there are chords in
  $C$.  
  
  Let $y''y_1$ be the chord from $P_{v',y}$ to $P_{wx,y}$ with $y''$ being
  the vertex closest to $v'$ and under this assumption $y_1$ is the vertex
  closest to $x$. Then the path formed by the subpath of the shortest path
  $P_{v,y}$ through $v'$ and $y''$, denoted as $P_{uv,y''}$ and the edge
  $y''y_1$, denoted as $P_{u,v',y_1}$ is an induced $u,y_1$-path through
  $v$. If $P_{u,v',y_1}$ is a shortest $u,y_1$-path, then this is a
  contradiction with minimality assumption, since $u,v,w,x$ and $y_1$
  provide a counter-example and $d(w,y_1)<d(w,y)$.  Thus we may assume that
  $P_{u,v',y_1}$ is not a shortest path and that $d(u,y_1)<d(v,y_1)+1$. Let
  $P_{u,y_1}$ be a shortest $u,y_1$-path.
  Now, $P_{u,y_1}$ must be vertex disjoint with $P_{u,v',y_1}$, since
  $d(u,y_1)<d(v,y_1)+1$.  Therefore the cycle $C_1$ formed by the union of
  the paths $P_{u,y_1}$ and $P_{u,v',y_1}$ is a cycle of length at least
  four. If $u$ and $y_1$ are adjacent, then the cycle formed by
  $P_{u,v',y_1}$ and the edge $uy_1$ is an induced cycle of length at least
  $4$, a contradiction with $G$ being chordal. Thus, $u$ and $y_1$ are not
  adjacent, and the length of $C_1$ is at least five. To avoid $C_1$ being
  an induced cycle, there are chords from $P_{u,y_1}$ to
  $P_{u,v',y_1}$. Consider the chord $u_1v_1$ with $u_1$ being the vertex
  closest to $u$, and under this assumption $v_1$ is the closest to $y''$
  in $P_{uv,y''}$. Then the subpath $P_{u,u_1}$ of $P_{u,y_1}$, $u_1v_1$
  and the subpath $P_{u,v_1}$ of $P_{uv,y''}$ form a cycle $C_2$ of length
  at least $4$ (note that $v=v'$ is possible). To avoid $C_2$ being
  induced, there should be chords $u_1x$ for all vertices $x$ in the
  subpath $P_{v,v_1}$ of the path $P_{u,v',y_1}$. In addition, $uu_1$ is an
  edge. Thus, if $C_2$ has at least $5$ vertices, we derive that there
  exists an induced $3$-fan with $u_1$ as its center, which is a
  contradiction with $G$ being Ptolemaic. Now, suppose $C_2$ has $4$
  vertices, Then, $v'=v$ and $v_1=y''$. Since $d(u,y_1)<d(v,y_1)+1 =3$, we
  infer that $u_1y_1$ is an edge. But then $u,v,y'',y_1$ and $u_1$ form a
  $3$-fan with $u_1$ as its center, again a contradiction. Thus, there are
  no chords between $P_{v',y}$ and $P_{wx,y}$.


  The case of chords from $P_{v',w}$ to $P_{wx,y}$ can be resolved in a
  similar way as in the previous paragraph, yielding that such chords are
  not possible.

  Finally, suppose that there are chords from $P_{v',w}$ to
  $P_{v',y}$. Among these chords, consider $w'y'$ such that $w'$ is the
  vertex on $P_{v',w}$ closest to $w$ with a chord and $y'$ being the
  vertex in $P_{v',y}$ closest to $y$ so that the path $P_{w,w'y'}$, formed
  by the union of the subpath $P_{w,w'}$ of $P_{v',w}$ and chord $w'y'$, is
  induced. The union of the paths $P_{w,w'y'}$, the subpath $P_{y',y}$ of
  $P_{v',y}$ and the path $P_{w,y}$ form a cycle $C'$ of length at least
  $5$. Since $C'$ cannot be an induced cycle in the Ptolemaic graph $G$,
  there must be chords from the $P_{w,w'}$ to $P_{w,y}$.  By the choice of
  $w'$, such a chord must connect $w'$ to a vertex in $P_{w,y}$. The last
  chord of this type that prevents an induced cycle of length greater than
  or equal to 4 connects $w'$ to $y$. This requires that $w'$ is adjacent
  to $w$ and $y$ is adjacent to $x$. Thus there is a cycle $C'' = ww' \cup
  w'y \cup P_{w,y}$. Since $C''$ must not be a long cycle, $w'x$ must be
  chord, but then the subgraph formed by the vertices $w,w',y',y$ and $x$
  should form an induced 3-fan, a contradiction. This scenario implies that
  $w'= v'$.  Now, the cycle formed by the union of the chord $v'x$, $xy$
  and $P_{v',y}$ will be an induced cycle of length greater than or equal
  to 4, unless $y'=y$.  The length of the path, say $P$, formed by the
  shortest path $P_{uv,v'}$ and $v'x$ is $k=d(u,w)\ge 2$.  Since the
  vertices $u,v,w,x,y$ form a counterexample to \AX{Sm}, we infer that
  $d(u,x)=k-1$. Therefore, there exists a $u,x$-shortest path, say
  $P_{u,x}$ of length $k-1$, not passing through $v$ and $v'$.  Let $u'$ be
  the neighbor of $u$ in $P_{u,x}$. The union of the paths formed by $P$
  and $P_{u,x}$ form a cycle, say $C_1$ of length greater than or equal to
  4. To prevent $C_1$ being an induced cycle, there must be chords from
  $u'$ to $P_{uv,v'}$. Since $P_{uv,v'}$ is a shortest path, the chords can
  only connect $u'$ to $v$ and to the vertex $v_1$ adjacent to $v$ on
  $P_{uv,v'}$. Now, to prevent the cycle formed by the union of $u'v_1$,
  the subpath, say $P_{v_1,v'x}$ of $P$ and the subpath $P_{u',x}$ of
  $P_{u,x}$, being an induced long cycle or a 3-fan, we infer that
  $v=v_1=v'$, $u'=x$ and thus $u,v,w,x,y$ form an induced $K_{1,1,3}$, a
  final contradiction.  
  \smallqed
  \smallskip
  \par\noindent The claim shows that $G$ has an induced cycle $C$ with at
  least $4$ vertices, which is contradiction with $G$ being chordal.
\end{proof}

\section{Smoothness among Weakly Modular Graphs}

Weakly modular graphs have received considerable attention in metric graph
theory, see e.g.\ \cite{bc-2008,chalopin2020weakly}, and it is therefore of
interest to consider smoothness in this much larger class of graphs.
\begin{definition}
\label{d:weak-mod}
A graph $G$ is \emph{weakly modular with respect to a vertex $u$} if its
distance function $d$ satisfies
\begin{itemize}
\item[] \emph{Triangle property:} For any two vertices $v,w$ with
  $1=d(v,w)<d(u,v)=d(u,w)$ there exists a common neighbor $z$ of $v$ and
  $w$ such that $d(u,x)=d(u,v)-1$.
\item[] \emph{Quadrangle property:} For any three vertices $v,w,z$ with
$d(v,z)=d(w,z)=1$ and  $2=d(v,w)\le d(u,v)=d(u,w)=d(u,z)-1,$ there
exists a common neighbor $x$ of $v$ and $w$ such that
$d(u,x)=d(u,v)-1$.
\end{itemize}
A graph $G$ is \emph{weakly modular} if it is weakly modular with respect
to any vertex $u$.
\end{definition}

A prominent subclass of weakly modular graphs are weakly median graps which
are exactly the weakly modular graphs that do not contain $K_{1,1,3}$,
$K_{2,3}$, the ``x-house'' $K_{1,1,3}^+$, and the wheel with a missing
spoke $W_4^-$ \cite{bandelt2000decomposition}. It was also shown in
\cite{bandelt2000decomposition} that $\ell_1$-graphs encompass all weakly
median graphs. Lemma~9 in \cite{Chepoi:94} establishes that
$U(v,u)=\{x\in V:\, v\in I[x,u]\}$ is convex for any two vertices $u$
and $v$ of a weakly median graph. Lemma~\ref{lem:U(a,b)-convex} thus yields
an alternative argument for the smoothness of weakly median graphs.

As shown in~\cite{Chung:89}, quasi-median graphs (which we have already
shown to be smooth since they are partial Hamming graphs) can be
characterized as the weakly modular graphs that are
$(K_{2,3},K_{4}-e)$-free. Since $K_{4}-e$ is an induced subgraph of
$K_{1,1,3}$, $W_4^-$, and $K_{1,1,3}^+$, they are a proper subclass of
weakly median graphs, which are smooth by virtue of being $\ell_1$-graphs.

Weakly modular $\ell_1$-graphs do not contain a induced $K_{2,3}$ or
$W_4^-$ \cite{chalopin2020weakly}, and $(K_{2,3},W_4^-)$-free weakly
modular graphs are known as premedian graphs \cite{Chastand:01}.
Lemma~4.24 in \cite{chalopin2020weakly}, moreover, asserts that all
$\ell_1$ weakly modular graphs are premedian graphs without propellers,
where a \emph{propeller}, or $K_5-K_3$, consists of three triangles glued
along a common edge, and hence are isomorphic to $K_{1,1,3}$. Since
$\ell_1$-graphs are smooth, it natural to ask whether all
$(K_{2,3},K_{1,1,3},W_4^-)$-free weakly modular graphs, i.e., the
$K_{1,1,3}$-free premedian graphs, are smooth. The graph in
Fig.~\ref{fig:g337} shows, however, that this is not the case.

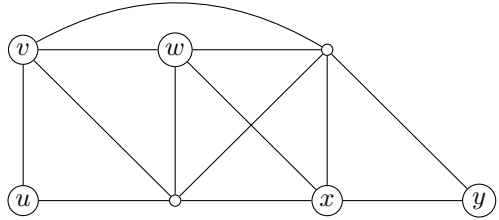
\begin{figure}[h]
  \centering
  \begin{tikzpicture}[every node/.style=
      {circle, draw, fill=white, inner sep=1.5pt}]
    \node (u) at (-6,0) {$u$};
    \node (v) at (-6,2) {$v$};
    \node (w) at (-4,2) {$w$};
    \node (a) at (-2,2) {};
    \node (b) at (-4,0) {};
    \node (x) at (-2,0) {$x$};
    \node (y) at (0,0) {$y$};
    \draw (u) -- (v) --(w) -- (a) -- (y) -- (x) --(b) -- (u);
    \draw (v) -- (b) -- (a) -- (x) --(w) -- (b);
    \draw (v) to[bend left=30] (a);
  \end{tikzpicture}
  \caption{A graph without induced $K_{2,3}$, $K_{1,1,3}$, and $W_4^-$.
    i.e., a $K_{1,1,3}$-free premedian graph, that is not smooth. This
    graph is also chordal, and thus also bridged, weakly bridged, and
    bucolic.  It is not a weakly median graph since it contains the x-house
    $K_{1,1,3}^+$, i.e., a $K_4$ with an attached triangle, as an induced
    subgraph. Five vertices violating \AX{Sm} as labeled.}
  \label{fig:g337}
\end{figure}

By Thm.~\ref{smooth-Ptolemaic}, the $K_{1,1,3}$-free Ptolemaic graphs are
smooth. These also form a subclass of weakly modular graph. More precisely,
they are a subclass of chordal graphs, and thus in particular bridged. A
graph is called \emph{bridged} if it does not contain isometric cycles
larger than triangles. Equivalently, they are the $(C_4,C_5)$-free weakly
modular graphs \cite{Chepoi:89}. The \emph{weakly bridged graphs} coincide
with the $C_4$-free weakly modular graphs \cite{Chepoi:15}.  The finite
bucolic graphs, introduced in \cite{bucolic}, can be characterized as the
$(K_{2,3},W_4,W_4^-)$-free weakly modular graphs.  Alternatively, finite
bucolic graphs can be obtained by gated amalgamation from Cartesian
products of 2-connected weakly bridged graphs \cite{bucolic}. Since both
$K_{2,3}$ and $W_4^-$ contain an induced $C_4$, bridged graphs and
therefore in particular Ptolemaic graphs are premedian graphs.  The x-house
$K_{1,1,3}^+$ is smooth and Ptolemaic, but not weakly median.  On the other
hand, $K_{1,1,3}$ is Ptolemaic. It is therefore of interest to consider the
$K_{1,1,3}$-free Ptolemaic graphs as a source of smooth graphs that are not
weakly median and also not $\ell_1$.

On the other hand, we may ask if there are larger well-studied subclasses
of $K_{1,1,3}$-free premedian graphs that are smooth.  Since $K_{1,1,3}$ is
bridged, the best we can hope for is that $K_{1,1,3}$-free bucolic,
weakly-bridged, or bridged graphs are smooth.  The graph in
Fig.~\ref{fig:g337}, however, is not smooth. It does not contain an induced
$K_{2,3}$, $W_4$, $W_4^-$, $C_4$, or $C_5$. Since it is premedian, it is in
particular a weakly modular graph, and hence also bucolic, weakly bridged,
and bridged. Each of these graph classes, therefore, contain non-smooth
graphs without an induced $K_{1,1,3}$.

\begin{figure}[htb]
  \begin{center}
    \centering
    \begin{tikzpicture}[every node/.style=
        {circle, draw, fill=white, inner sep=1.5pt}]
      \node (0) at (2, 0) {$w$};        
      \node (1) at (0,3.5) {$u$};
      \node (2) at (-2,0) {$y$};
      \node (3) at (0,2) {$v$};
      \node (4) at (0,0) {$x$};
      \node (5) at (-2,2) {};
      \node (6) at (2, 2) {};
      \draw (0) -- (4) -- (2) -- (5) -- (3) -- (6) -- (0);
      \draw (0) -- (3) -- (2);
      \draw (6) -- (4) -- (5);
      \draw (6) -- (1) -- (5);
      \draw [line width=.4mm] (1) -- (3);
      \draw [line width=.4mm] (1) -- (5);
      \draw [line width=.4mm] (1) -- (6);
      \draw [line width=.4mm] (3) -- (5);
      \draw [line width=.4mm] (3) -- (6);
      \draw [line width=.4mm] (4) -- (5);
      \draw [line width=.4mm] (4) -- (6);
    \end{tikzpicture}
  \end{center}
  \caption{Example of $(K_{1,1,3}, K_{2,3}, K_{1,1,3}^+)$-free weakly
    modular graph that is not smooth. Note that this graph contains an
    induced $W_4^-$ (bold edges). The vertices obstructing \AX{Sm} are
    labelled.}
  \label{fig:onlyW4-}
\end{figure}
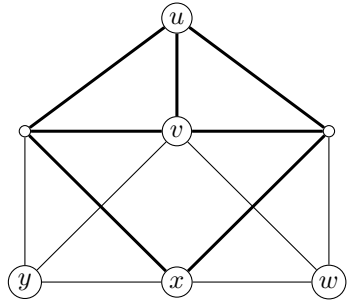

According to \cite{bandelt2000decomposition}, the
$(K_{1,1,3},K_{1,1,3}^+)$-free bridged graphs are exactly the weakly median
bridged graphs. Hence this subclass of bridged graphs is smooth.  In
particular, the counter example in Fig.~\ref{fig:g337} contains an induced
``x-house'' $K_{1,1,3}^+$. This graph is also the only non-smooth
$(K_{1,1,3}, K_{2,3}, W_4^-)$-free graph with at most $8$ vertices.  A
computational survey of bridged graphs without $K_{1,1,3}$ and up to 8
vertices showed that they all contain a $K_{1,1,3}^+$. Hence we reasonably
may ask whether the $(K_{1,1,3},K_{1,1,3}^+,K_{2,3})$-free weakly modular
graphs are smooth. A computational survey, however, identified several
examples of such graphs that are not smooth, one of which is shown in
Fig.~\ref{fig:onlyW4-}.  Interestingly, they contain an induced
$W_4^-$. Taken together, these negative results suggest that the smooth
graphs do not extend to any of the ``obvious'' generalizations of weakly
median graphs among the weakly modular graphs.

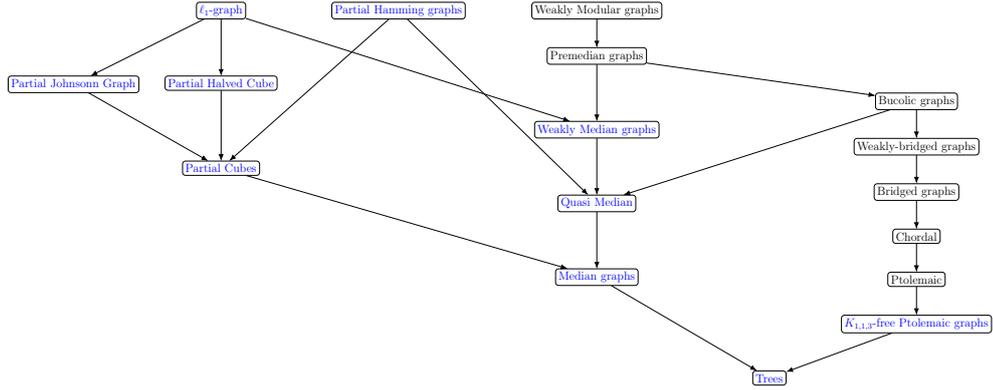
\begin{figure}[h]
\centering
\resizebox{0.99\textwidth}{!}{%
  \begin{tikzpicture}[
      every node/.style={draw, rounded corners=3pt, fill=white!10, inner
        sep=3pt, font=\large}, edge/.style={-Latex, thick} ]    
    \node (weakly-mod) {Weakly Modular graphs};
    \node(l1-graph) [left =10cm of weakly-mod]
         {\textcolor{blue}{$\ell_1$-graph}};
    \node (partial-hamming) [right =3cm of l1-graph]{\textcolor{blue}{Partial Hamming graphs}};     
\node (pre-median) [below =1cm of weakly-mod]{Premedian graphs};
\node (weakly-med) [below =2cm  of pre-median]{\textcolor{blue}{Weakly Median graphs}};
\node (partial-cube) [below=5cm of l1-graph]{\textcolor{blue}{Partial Cubes}};
\node (quasi-median)[below =2cm of weakly-med]{\textcolor{blue}{Quasi Median}};
\node (med) [below=2cm of quasi-median]{\textcolor{blue}{Median graphs}};
\node (Bucolic) [below right=1cm and 8cm of pre-median]{Bucolic graphs};
\node (weakly-bridged) [below =1cm of Bucolic]{Weakly-bridged graphs};
\node (bridged) [below =1cm of weakly-bridged]{Bridged graphs};
\node (chordal) [below =1cm of bridged]{Chordal};
\node (ptolemaic) [below =1cm of chordal]{Ptolemaic};
\node (ptolemaic-K113)[below =1cm of
  ptolemaic]{\textcolor{blue}{$K_{1,1,3}$-free Ptolemaic graphs}};

  \node(partial-johnsonn)[below left =2cm and 2cm of l1-graph]{\textcolor{blue}{Partial Johnsonn Graph}};
  \node(partial-halved)[below =2cm  of l1-graph]{\textcolor{blue}{Partial Halved Cube}};
  \node (tree) [below right =3cm and 4cm of med]{\textcolor{blue}{Trees}};

\draw[edge] (weakly-mod) -- (pre-median);
\draw[edge] (pre-median) -- (weakly-med);
\draw[edge] (pre-median) -- (Bucolic);
\draw[edge] (Bucolic) -- (weakly-bridged);
\draw[edge] (weakly-bridged) -- (bridged);
\draw[edge] (bridged) -- (chordal);
\draw[edge] (chordal) -- (ptolemaic);
\draw[edge] (weakly-med) -- (quasi-median);
\draw[edge] (quasi-median) -- (med);
\draw[edge] (partial-hamming) -- (quasi-median);
\draw[edge] (partial-hamming) -- (partial-cube);
\draw[edge]  (partial-cube) -- (med);
\draw[edge] (Bucolic) -- (quasi-median);
\draw[edge] (ptolemaic) -- (ptolemaic-K113);
\draw[edge] (l1-graph) -- (partial-johnsonn);
\draw[edge] (l1-graph) -- (partial-halved);
\draw[edge] (partial-johnsonn) -- (partial-cube);
\draw[edge] (partial-halved) -- (partial-cube);
\draw[edge] (med) -- (tree);
\draw[edge] (ptolemaic-K113) -- (tree); 
\draw[edge] (l1-graph) -- (weakly-med);
\end{tikzpicture}}
\caption{Hierarchy of graph classes related to Smooth Axioms, where the
  blue colored families are smooth graphs.}
\end{figure}

On the other hand, $K_{1,1,3}^+$ is a pre-median bridged graph, but not
weakly median, and nevertheless smooth. Similarly, the graph $W_4^-$ is smooth.
Therefore, there are weakly modular graphs that are smooth and not even
pre-median, and in particular not $\ell_1$-graphs.

A graph class of interest in this context are the \emph{pseudo-modular}
graphs \cite{Bandelt:86}, i.e, graphs in which all metric triangles have
size $0$ or $1$. They are weakly modular graphs characterized by the metric
condition: if $1\le d(u,w)\le 2$ and $d(v,u)=d(v,w)=k\ge 2$ then there is a
vertex $x$ with $d(u,x)=d(w,x)=1$ and $d(v,x)=k-1$. The counterexample in
Fig.~\ref{fig:onlyW4-} is pseudo-modular, suggesting that there is no
straightforward relationship between smoothness and pseudomodularity.

\section{Open Questions}

A comparison of graph classes that we have shown to be smooth, and the
  subclasses of weakly modular graphs immediately suggests to consider the
  following challenge:
\begin{problem}
  Characterize the smooth weakly modular graphs as a subclass of the
  $(K_{2,3},K_{1,1,3})$-free weakly modular graphs.
\end{problem}

In metric graph theory concerned with median-like classes of graphs, the
fact that a class of graphs is closed under the Cartesian product and the
gated amalgamation operations, opens the question of describing the prime
building blocks from which this class of graphs can be generated;
cf.~\cite{bc-2008,bresar2003}. Here we have proved that the strong product
of smooth graphs are also smooth. Based on
Theorems~\ref{thm:smoothproduct}, \ref{gated-amalgam-smooth} and
\ref{thm:strong}, it seems natural to modify the notion of prime graphs such
that a graph is \emph{``strongly prime''} if it is neither a gated amalgam of
any proper subgraphs nor a nontrivial Cartesian or strong product. 
\begin{problem}
Characterize the ``strongly prime'' smooth graphs.
\end{problem}
In \cite{bandelt2000decomposition}, it is established that the set
$\mathbf{P}_{WM}$ of prime weakly median graphs comprises precisely (i) the
5-wheel $W_5$ (a 5-cycle plus a pivot vertex adjacent to all vertices of
the cycle), (ii) the subhyperoctahedra (induced subgraphs of
hyperoctahedra, i.e., cocktail-party graphs), that is, multipartite graphs
of the form $K_{i_1,i_2,\ldots}$ with $1 \leq i_j \leq 2$ different from
the singleton graph $K_1$, the 3-vertex path $P_2 = K_{1,2}$, and the
4-cycle $C_4 = K_{2,2}$, and (iii) the 2-connected $(K_4,K_{1,1,3})$-free
bridged graphs. Since weakly median graphs are smooth and preserve the
operation of Cartesian products and gated amalgamations, it follows that
these graphs are also prime graphs for the smooth graphs. Now consider the
family $\mathcal{F}$ of all graphs obtained by performing the operations of
Cartesian products, strong products and gated amalgamations starting from
the set of prime complete graphs $\{K_{p}:\, p \text{ is a prime
  integer}\}$ and set of prime graphs $\mathbf{P}_{WM}$ for weakly median
graphs.  By Theorems~\ref{thm:smoothproduct}, \ref{thm:strong} and
\ref{gated-amalgam-smooth}, all graphs in $\mathcal{F}$ are smooth. This
leads to the following question:
\begin{problem}
  Is $\mathcal{F}$ the family of all smooth weakly modular graphs?
\end{problem}
It will also be of interest to explore whether smooth graphs are preserved by isometric expansion \cite{Chepoi:88}.

The fact that a large class of graphs is smooth suggests that \AX{Sm} may
also be of interest in a broader context, i.e., beyond the realm of
graphs. The equivalence of \AX{Sm} and \AX{Sm*}, Lemma~\ref{lem:interval2},
suggests to consider smooth transit functions\footnote{$R: X\times X\to
2^X$ such that $x,y\in R(x,y)$, $R(x,y)=R(y,x)$, and $R(x,x)=\{x\}$ for all
$x,y\in X$} \cite{Mulder:80} and in particular smooth geometric transit
functions \cite{Nebesky:01}, which in addition to \AX{Sm*} also satisfy the
two of the betweenness axioms \AX{b2} and \AX{b3}. Moreover, the smoothness
property expressed in the form of distances in Obs.~\ref{obs:dist} appears
to be of interest also as property of (finite) metric spaces e.g.\ in the
context of phylogenetic networks.

\begin{small}
\section*{Declarations}

\subsection*{Availability of Data and Materials}
There are no data associated with this work.

\subsection*{Competing interests}
The authors declare that they have no competing interests, or other
interests that might be perceived to influence the results and/or
discussion reported in this paper.

\subsection*{Dedication}
This contribution is dedicated to the memory of Andreas W.\ M.\ Dress.

\subsection*{Acknowledgment}
We thank Victor Chepoi for pointing out the connection between
  smoothness and convexity of the point-shadow sets.

\subsection*{Funding}
This work was supported in part by the DST, Govt. of India (Grant
No. DST/INT/DAAD/P-03/ 2023), the DAAD, Germany (Grant No. 57683501), PFS
acknowledges the financial support by the Federal Ministry of Research,
Technology and Space of Germany (BMFTR) through DAAD project 57616814
(SECAI, School of Embedded Composite AI), and jointly with the
S{\"a}chsische Staatsministerium f{\"u}r Wissenschaft, Kultur und Tourismus
in the programme Center of Excellence for AI-research \emph{Center for
Scalable Data Analytics and Artificial Intelligence Dresden/Leipzig},
project identification number: SCADS24B. B.B. was supported by the
Slovenian Research and Innovation agency (grants P1-0297, J1-70045, N1-0285, and
N1-0431).

\subsection*{Authors' contributions}
MC and PFS designed the study. All authors contributed to the
mathematical results and the drafting of the manuscript. 
\end{small}

\bibliography{smoothbib}
\end{document}